\documentclass[12pt]{amsart}
\usepackage{amsmath,amsthm,amssymb}
\usepackage{graphicx}
\usepackage{subfig}

\DeclareMathOperator{\RE}{Re} \DeclareMathOperator{\IM}{Im}

\numberwithin{equation}{section}
\newtheorem{theorem}{Theorem}[section]
\newtheorem{lemma}[theorem]{Lemma}
\newtheorem{corollary}[theorem]{Corollary}
\theoremstyle{remark}
\newtheorem{remark}[theorem]{Remark}

\newtheorem{definition}[theorem]{Definition}

\begin{document}

\title{Construction of subclasses of univalent harmonic mappings}

\thanks{The research work of the first author is supported by research fellowship from Council of Scientific and
Industrial Research (CSIR), New Delhi.}

\author[S. Nagpal]{Sumit Nagpal}
\address{Department of Mathematics, University of Delhi,
Delhi--110 007, India}
\email{sumitnagpal.du@gmail.com }

\author[V. Ravichandran]{V. Ravichandran}

\address{Department of Mathematics, University of Delhi,
Delhi--110 007, India \and
School of Mathematical Sciences,
Universiti Sains Malaysia, 11800 USM, Penang, Malaysia}
\email{vravi68@gmail.com}

\begin{abstract}
Complex-valued harmonic functions that are univalent and sense-preserving in the open unit disk are widely studied. A new methodology is employed to construct subclasses of univalent harmonic mappings from a given subfamily of univalent analytic functions. The notion of harmonic Alexander integral operator is introduced. Also, the radius of convexity for certain families of harmonic functions is determined.
\end{abstract}



\maketitle

\section{introduction}
Let $\mathcal{H}$ denote the class of all complex-valued harmonic functions $f$ in the unit disk $\mathbb{D}=\{z \in \mathbb{C}:|z|<1\}$ normalized by $f(0)=0=f_{z}(0)-1=f_{\bar{z}}(0)$. Such functions can be written in the form $f=h+\bar{g}$, where
\begin{equation}\label{eq1.1}
h(z)=z+\sum_{n=2}^{\infty}a_{n}z^{n}\quad\mbox{and}\quad g(z)=\sum_{n=2}^{\infty}b_{n}z^n
\end{equation}
are analytic in $\mathbb{D}$. In 1984, Clunie and Sheil-Small \cite{cluniesheilsmall} investigated the subclass $\mathcal{S}_{H}^{0}$ of $\mathcal{H}$ consisting of univalent and sense-preserving functions, that is, for a function $f=h+\bar{g} \in \mathcal{S}_{H}^{0}$, $f$ is one-to-one in $\mathbb{D}$ and the Jacobian $J_{f}(z)=|h'(z)|^2-|g'(z)|^2$ is positive or equivalently $|g'(z)|<|h'(z)| $ in $\mathbb{D}$. The class $\mathcal{S}_{H}^{0}$ is a compact family with respect to the topology of locally uniform convergence. The classical family $\mathcal{S}$ of normalized analytic univalent functions is a subclass of $\mathcal{S}_{H}^{0}$. Let $\mathcal{S}_{H}^{*0}$, $\mathcal{K}_{H}^{0}$ and $\mathcal{C}_{H}^{0}$ be the subclasses of $\mathcal{S}_{H}^{0}$ mapping $\mathbb{D}$ onto starlike, convex and close-to-convex domains, respectively, just as $\mathcal{S}^{*}$, $\mathcal{K}$ and $\mathcal{C}$ are the subclasses of $\mathcal{S}$ mapping $\mathbb{D}$ onto their respective domains.

In \cite{sumit1} the authors investigated the properties of functions in the subclass $\mathcal{F}_{H}^{0} \subset \mathcal{C}_{H}^{0}$ defined by the condition $|f_{z}(z)-1|<1-|f_{\bar{z}}(z)|$ for all $z \in \mathbb{D}$. This subclass was closely related to the class $\mathcal{F} \subset \mathcal{C}$ introduced by MacGregor \cite{macgregor2} consisting of analytic functions which satisfy $|f'(z)-1|<1$ for $z \in \mathbb{D}$. The authors proved that a harmonic function $f=h+\bar{g} \in \mathcal{F}_{H}^{0}$ if and only if the analytic functions $h+\epsilon g$ belong to $\mathcal{F}$ for each $|\epsilon|=1$. Using this property, authors established the coefficient estimates, growth results, boundary behavior, convolution properties and sharp bound for radius of convexity and starlikeness for the class $\mathcal{F}_{H}^{0}$. This connection between the classes $\mathcal{F}$ and $\mathcal{F}_{H}^{0}$ has motivated to give the following definition which turns out to be a simple but an effective method in construction of subclasses of univalent harmonic mappings from a given subfamily of $\mathcal{S}$.

\begin{definition}\label{def}
Suppose that $\mathcal{G}$ is a subfamily of $\mathcal{S}$. Denote by $\mathcal{G}_{H}^{0}$ the class consisting of harmonic functions $f=h+\bar{g}$ for which $h+\epsilon g \in \mathcal{G}$ for each $|\epsilon|=1$, $h$ and $g$ being analytic functions in $\mathbb{D}$. We call $\mathcal{G}_{H}^{0}$ the \textbf{harmonic analogue} of $\mathcal{G}$ and write $\mathcal{G}\triangleright\mathcal{G}_{H}^{0}$.
\end{definition}

By Definition \ref{def}, it readily follows that $\mathcal{F}\triangleright \mathcal{F}_{H}^{0}$. If $\mathcal{G}_{H}^{0}$ is the harmonic analogue of $\mathcal{G}\subset \mathcal{S}$ then it is easy to see that $\mathcal{G}\subset \mathcal{G}_{H}^{0}$. Further properties of the harmonic analogue $\mathcal{G}_{H}^{0}$ for a subfamily $\mathcal{G}\subset \mathcal{S}$ are investigated in Section \ref{sec2}. In Section \ref{sec3}, the harmonic analogues of some well-known subclasses of $\mathcal{S}$ are determined and their respective properties are discussed.

Let $\mathcal{A}$ be the subclass of $\mathcal{H}$ consisting of normalized analytic functions. Let $\Lambda:\mathcal{A}\rightarrow \mathcal{A}$ be the Alexander integral operator \cite{alexander} defined by
\begin{equation}\label{eq1.2}
\Lambda[f](z)=\int_{0}^{z}\frac{f(t)}{t}\,dt.
\end{equation}
A well-known classical result implies that $\Lambda$ does not carry $\mathcal{S}$ into $\mathcal{S}$ (see \cite[Chapter 14]{goodman}). In the last section of this paper, the notion of a harmonic Alexander operator $\Lambda_{H}:\mathcal{H}\rightarrow\mathcal{H}$ is introduced and its properties are investigated. The radius of convexity for certain families of harmonic functions is also determined.

\section{Harmonic Analogue $\mathcal{G}_{H}^{0}$}\label{sec2}
In this section, we will investigate the properties of the harmonic analogue $\mathcal{G}_{H}^{0}$ for a subfamily $\mathcal{G}\subset \mathcal{S}$. For this, it will be necessary to discuss the notion of stable harmonic mappings introduced by Hern\'{a}ndez and Mart\'{i}n  in \cite{stable}. A sense-preserving harmonic mapping $f=h+\bar{g}$ is said to be stable univalent (resp. stable starlike, stable convex and stable close-to-convex) if all the mappings $f_{\lambda}=h+\lambda \bar{g}$ with $|\lambda|=1$ are univalent (resp. starlike, convex and close-to-convex) in $\mathbb{D}$. The following result was proved in \cite{stable}.

\begin{lemma}\label{lem}
A sense-preserving harmonic mapping $f=h+\bar{g}$ is stable univalent (resp. stable starlike, stable convex and stable close-to-convex) if and only if the analytic functions $F_{\lambda}=h+\lambda g$ are univalent (resp. starlike, convex and close-to-convex) in $\mathbb{D}$ for each $|\lambda|=1$.
\end{lemma}

Let $\mathcal{SS}_{H}^{0}$, $\mathcal{SS}_{H}^{*0}$, $\mathcal{SK}_{H}^{0}$ and $\mathcal{SC}_{H}^{0}$ be subclasses of $\mathcal{S}_{H}^{0}$ consisting of stable univalent, stable starlike, stable convex and stable close-to-convex mappings respectively. Then $\mathcal{S}^{*}\subset \mathcal{SS}_{H}^{*0} \subset \mathcal{S}_{H}^{*0}$, $\mathcal{K}\subset \mathcal{SK}_{H}^{0} \subset \mathcal{K}_{H}^{0}$ and $\mathcal{C}\subset \mathcal{SC}_{H}^{0} \subset \mathcal{C}_{H}^{0}$. Moreover $\mathcal{SK}_{H}^{0}\subset \mathcal{SS}_{H}^{*0} \subset \mathcal{SC}_{H}^{0} \subset \mathcal{SS}_{H}^{0}$. In view of Definition \ref{def} and Lemma \ref{lem}, it is easy to deduce that $\mathcal{SS}_{H}^{0}$, $\mathcal{SS}_{H}^{*0}$, $\mathcal{SK}_{H}^{0}$ and $\mathcal{SC}_{H}^{0}$ are harmonic analogues of $\mathcal{S}$, $\mathcal{S}^{*}$, $\mathcal{K}$ and $\mathcal{C}$ respectively

The first theorem is quite simple but a useful tool in the investigation of results regarding the harmonic analogue $\mathcal{G}_{H}^{0}$ for a subfamily $\mathcal{G}\subset \mathcal{S}$.
\begin{theorem}\label{th2.2}
Suppose that $\mathcal{G}\subset \mathcal{S}$ and $\mathcal{G}\triangleright \mathcal{G}_{H}^{0}$. Then
\begin{itemize}
  \item [$(i)$] $\mathcal{G}_{H}^{0} \subset \mathcal{SS}_{H}^{0}$;
  \item [$(ii)$] If $f \in \mathcal{S}\cap \mathcal{G}_{H}^{0}$, then $f \in \mathcal{G}$;
  \item [$(iii)$] If $f=h+\bar{g} \in \mathcal{G}_{H}^{0}$, then the harmonic mappings $f_{\lambda}=h+\lambda \bar{g} \in \mathcal{G}_{H}^{0}$ for each $|\lambda|=1$;
  \item [$(iv)$] If $\mathcal{J} \subset \mathcal{G}$, then $\mathcal{J}_{H}^{0} \subset \mathcal{G}_{H}^{0}$ where $\mathcal{J}_{H}^{0}$ is the harmonic analogue of $\mathcal{J}$.
\end{itemize}
\end{theorem}

\begin{proof}
Let $f=h+\bar{g} \in \mathcal{G}_{H}^{0}$. Then $h+\epsilon g \in \mathcal{G}$ for each $|\epsilon|=1$ which imply that $h(0)=g(0)=h'(0)-1=g'(0)$ using the normalization of $\mathcal{G}$. By Lemma \ref{lem} it follows that $f \in \mathcal{SS}_{H}^{0}$ since $h+\epsilon g$ is univalent for each $|\epsilon|=1$. This proves $(i)$. The part $(ii)$ follows immediately from Definition \ref{def}. Regarding the proof of $(iii)$, if $f=h+\bar{g} \in \mathcal{G}_{H}^{0}$ and $|\lambda|=1$, then it is easy to see that $h+\lambda \epsilon g \in \mathcal{G}$ for each $|\epsilon|=1$ so that $h+\lambda \bar{g} \in \mathcal{G}_{H}^{0}$. To prove $(iv)$, let $f=h+\bar{g} \in \mathcal{J}_{H}^{0}$. As $\mathcal{J}\triangleright \mathcal{J}_{H}^{0}$, $h+\epsilon g \in \mathcal{J}$ for each $|\epsilon|=1$. But $\mathcal{J} \subset \mathcal{G}$ so that $h+\epsilon g \in \mathcal{G}$ for each $|\epsilon|=1$ which shows that $f \in \mathcal{G}_{H}^{0}$. This completes the proof of the theorem.
\end{proof}

Theorem \ref{th2.2}$(ii)$ conveys that every analytic univalent function in $\mathcal{G}_{H}^{0}$ is a member of $\mathcal{G}$. Since the members of $\mathcal{G}_{H}^{0}$ are stable univalent by Theorem \ref{th2.2}$(i)$, we have the following corollary which follows by \cite[Theorem 4, p. 17]{stable}.

\begin{corollary}\label{cor2.3}
Suppose that $\mathcal{G}\subset \mathcal{S}$ and $\mathcal{G}\triangleright \mathcal{G}_{H}^{0}$. If $f=h+\bar{g} \in \mathcal{G}_{H}^{0}$, then the analytic mappings $F_{\mu}=h+\mu g$ are univalent in $\mathbb{D}$ for each $|\mu|\leq 1$. In particular, $h$ is univalent.
\end{corollary}

Recall that convexity and starlikeness are hereditary properties for conformal mappings which do not extend to harmonic mappings (see \cite{duren}). Chuaqui, Duren and Osgood \cite{chuaqui} introduced the notion of fully starlike and fully convex functions that do inherit the properties of starlikeness and convexity respectively (see also \cite{sumit2}). A harmonic mapping $f$ of the unit disk $\mathbb{D}$ is fully convex if it maps every circle $|z|=r<1$ in a one-to-one manner onto a convex curve. Such a harmonic mapping $f$ with $f(0)=0$ is fully starlike if it maps every circle $|z|=r<1$ in a one-to-one manner onto a curve that bounds a domain starlike with respect to the origin. Applying Theorem \ref{th2.2}$(iv)$ and using the fact that stable starlike (resp. stable convex) mappings are fully starlike (resp. fully convex) (see \cite{stable,sumit2}), we have

\begin{corollary}\label{cor2.4}
Suppose that $\mathcal{G}\subset \mathcal{S}$ and $\mathcal{G}\triangleright \mathcal{G}_{H}^{0}$. If $\mathcal{G} \subset \mathcal{S}^{*}$ (resp. $\mathcal{G} \subset \mathcal{K}$), then members of $\mathcal{G}_{H}^{0}$ are fully starlike (resp. fully convex) in $\mathbb{D}$.
\end{corollary}

It is easy to see that if $\mathcal{I}$ and $\mathcal{J}$ are subclasses of $\mathcal{S}$ with $\mathcal{I}\triangleright\mathcal{I}_{H}^{0}$ and $\mathcal{J}\triangleright\mathcal{J}_{H}^{0}$, then $\mathcal{I}\cap\mathcal{J}\triangleright\mathcal{I}_{H}^{0}\cap\mathcal{J}_{H}^{0}$ and $\mathcal{I}\cup\mathcal{J}\triangleright\mathcal{I}_{H}^{0}\cup\mathcal{J}_{H}^{0}$. The next theorem determines the coefficient bounds for functions in the harmonic analogue $\mathcal{G}_{H}^{0}$.

\begin{theorem}\label{th2.5}
Suppose that $\mathcal{G}\subset \mathcal{S}$ and $\mathcal{G}\triangleright \mathcal{G}_{H}^{0}$. Let the Taylor coefficients $a_{n}(f)$ of the series of each $f \in \mathcal{G}$ satisfies $|a_n(f)|\leq p(n)$ for $n=2,3,\ldots$ where $p$ is a function of $n$. Then
\begin{itemize}
  \item [$(a)$] The respective Taylor coefficients $A_n(f)$ and $B_n(f)$ of the series of $h$ and $g$ of each function $f=h+\bar{g} \in \mathcal{G}_{H}^{0}$ satisfies
      \[||A_n(f)|-|B_n(f)||\leq p(n),\quad n=2,3,\ldots.\]
  \item [$(b)$] Let $h_0 \in \mathcal{G}$ be such that its Taylor coefficients satisfy $|a_n(h_0)|=p(n)$ for $n=2,3,\ldots$. Then for an analytic function $g_0$, the harmonic function $f_0=h_0+\bar{g}_0 \in \mathcal{G}_{H}^{0}$ if and only if $g_0\equiv 0$.
\end{itemize}
\end{theorem}

\begin{proof}
Let $f=h+\bar{g} \in \mathcal{G}_{H}^{0}$. Then $h+\epsilon g \in \mathcal{G}$ for each $|\epsilon|=1$ so that $|a_n(h+\epsilon g)|\leq p(n)$ for $n=2,3,\ldots$. But $a_n(h+\epsilon g)=A_n(f)+\epsilon B_n(f)$ for $n=2,3,\ldots$, which proves $(a)$.

Regarding $(b)$, suppose that $f_0=h_0+\bar{g}_0 \in \mathcal{G}_{H}^{0}$. By the proof of part $(a)$ it is easy to deduce that
\[|A_n(f_0)|+|B_n(f_0)|\leq p(n) \quad \mbox{for} \quad n=2,3,\ldots.\]
But $|A_n(f_0)|=|a_n(h_0)|=p(n)$ for $n=2,3,\ldots$ so that $B_n(f_0)=0$ for $n=2,3,\ldots$. Thus $g_0 \equiv 0$. The converse part is obvious.
\end{proof}

The next theorem determines the upper and lower bounds on the growth of a harmonic mapping in $\mathcal{G}_{H}^{0}$.
\begin{theorem}\label{th2.6}
Suppose that $\mathcal{G}\subset \mathcal{S}$ and $\mathcal{G}\triangleright \mathcal{G}_{H}^{0}$. If
\[P(|z|) \leq |f'(z)| \leq Q(|z|),\quad z\in \mathbb{D}\]
for each $f \in \mathcal{G}$ where $P$ and $Q$ are integrable functions of $|z|$, then each $f \in \mathcal{G}_{H}^{0}$ satisfies
\[\int_{0}^{|z|} P(\rho)\, d\rho \leq |f(z)|\leq \int_{0}^{|z|} Q(\rho)\, d\rho, \quad z \in \mathbb{D}.\]
In particular, the range of every function $f \in \mathcal{G}_{H}^{0}$ contains the disk
\[\left\{w \in \mathbb{C}:|w|<\lim_{|z|\rightarrow 1} \int_{0}^{|z|} P(\rho)\, d\rho\right\},\]
provided the limit exists.
\end{theorem}

\begin{proof}
Let $f=h+\bar{g} \in \mathcal{G}_{H}^{0}$. Then $h+\epsilon g \in \mathcal{G}$ for each $|\epsilon|=1$ so that
\[P(|z|)\leq |h'(z)+\epsilon g'(z)|\leq Q(|z|), \quad z \in \mathbb{D}.\]
In particular, this shows that
\[P(|z|)\leq |h'(z)|-|g'(z)|\quad \mbox{and}\quad |h'(z)|+|g'(z)|\leq Q(|z|), \quad z \in \mathbb{D}.\]
If $\Gamma$ is the radial segment from $0$ to $z$, then
\[|f(z)|=\left|\int_{\Gamma} \frac{\partial f}{\partial\zeta}\,d\zeta+\frac{\partial f}{\partial\overline{\zeta}}\,d\overline{\zeta}\right|\leq \int_{\Gamma} (|h'(\zeta)|+|g'(\zeta)|)|d\zeta|\leq \int_{0}^{|z|} Q(\rho)\,d\rho.\]
Next, let $\Gamma$ be the pre-image under $f$ of the radial segment from $0$ to $f(z)$. Then
\[|f(z)|=\int_{\Gamma}\left| \frac{\partial f}{\partial\zeta}\,d\zeta+\frac{\partial f}{\partial\overline{\zeta}}\,d\overline{\zeta}\right|\geq \int_{\Gamma} (|h'(\zeta)|-|g'(\zeta)|)|d\zeta|\geq \int_{0}^{|z|} P(\rho)\,d\rho.\qedhere\]
\end{proof}

\begin{theorem}\label{th2.7}
Suppose that $\mathcal{G}\subset \mathcal{S}$ and $\mathcal{G}\triangleright \mathcal{G}_{H}^{0}$. Then $\mathcal{G}$ is compact if and only if $\mathcal{G}_{H}^{0}$ is compact.
\end{theorem}

\begin{proof}
For necessary part, suppose that $f_n=h_n+\overline{g}_n \in \mathcal{G}_{H}^{0}$ for $n=1,2,\ldots$ and that $f_n\rightarrow f$ uniformly on compact subsets of $\mathbb{D}$. Then $f$ is harmonic and so $f=h+\bar{g}$. It is easy to see that $h_n\rightarrow h$ and $g_n\rightarrow g$ locally uniformly so that $h_n+\epsilon g_n \rightarrow h+\epsilon g$ for each $|\epsilon|=1$. Since $h_n+\epsilon g_n \in \mathcal{G}$ it follows that $h+\epsilon g \in \mathcal{G}$ for each $|\epsilon|=1$ using the compactness of $\mathcal{G}$. Thus $f=h+\bar{g} \in \mathcal{G}_{H}^{0}$.

For sufficient part, let $f_n \in \mathcal{G}$ for $n=1,2,\ldots$ such that $f_n \rightarrow f$ uniformly on compact subsets of $\mathbb{D}$. Then $f$ is univalent. Since  $\mathcal{G} \subset \mathcal{G}_{H}^{0}$ and $\mathcal{G}_{H}^{0}$ is compact, $f \in \mathcal{G}_{H}^{0}$. By Theorem \ref{th2.2}$(ii)$, $f \in \mathcal{G}$. This completes the proof.
\end{proof}

The next theorem investigates the relation between the radius of starlikeness, convexity and close-to-convexity of the classes $\mathcal{G}$ and $\mathcal{G}_{H}^{0}$.

\begin{theorem}\label{th2.8}
Suppose that $\mathcal{G}\subset \mathcal{S}$ and $\mathcal{G}\triangleright \mathcal{G}_{H}^{0}$. Then the classes $\mathcal{G}$ and $\mathcal{G}_{H}^{0}$ have the same radius of starlikeness, convexity and close-to-convexity.
\end{theorem}

\begin{proof}
Since $\mathcal{G} \subset \mathcal{G}_{H}^{0}$ it suffices to show that if $r_{0}$ is the radius of starlikeness (resp. convexity and close-to-convexity) of $\mathcal{G}$ then $f$ is starlike (resp. convex and close-to-convex) in $|z|<r_{0}$ for each $f \in \mathcal{G}_{H}^{0}$. To see this, suppose that $f=h+\bar{g} \in \mathcal{G}_{H}^{0}$. Then the analytic functions $h+\epsilon g$ belong to the class $\mathcal{G}$. Consequently the functions $h+\epsilon g$ are starlike (resp. convex and close-to-convex) in $|z|<r_{0}$. In view of Lemma \ref{lem}, it follows that $f$ is starlike (resp. convex and close-to-convex) in $|z|<r_{0}$.
\end{proof}

For analytic functions
\begin{equation}\label{eq2.1}
f(z)=z+\sum_{n=2}^{\infty}a_{n}z^{n}\quad \mbox{and}\quad F(z)=z+\sum_{n=2}^{\infty}A_{n}z^{n}
\end{equation}
belonging to $\mathcal{A}$, their convolution (or Hadamard product) is defined as
\[(f*F)(z)=z+\sum_{n=2}^{\infty}a_{n}A_{n}z^{n}, \quad z \in \mathbb{D}.\]
In the harmonic case, with $f=h+\bar{g}$ and $F=H+\bar{G}$ belonging to $\mathcal{H}$, their harmonic convolution is defined as
\[f*F=h*H+\overline{g*G}.\]
Harmonic convolutions was investigated in \cite{cluniesheilsmall,dorff1,dorff2,goodloe,ruscheweyh}.

Suppose that $\mathcal{I}$ and $\mathcal{J}$ are subclasses of $\mathcal{H}$. We say that a class $\mathcal{I}$ is closed under convolution if $\mathcal{I}*\mathcal{I}\subset \mathcal{I}$, that is, if $f$, $g \in \mathcal{I}$ then $f*g \in \mathcal{I}$. Similarly, the class $\mathcal{I}$ is closed under convolution with members of $\mathcal{J}$ if $\mathcal{I}*\mathcal{J}\subset \mathcal{I}$. The next theorem discusses the convolution properties of the class $\mathcal{G}_{H}^{0}$.

\begin{theorem}\label{th2.9}
Suppose that $\mathcal{G}\subset \mathcal{S}$ is closed under convolution and $\mathcal{G}\triangleright \mathcal{G}_{H}^{0}$. Then
\begin{itemize}
  \item [$(i)$] The convolution of each member of $\mathcal{G}_{H}^{0}$ with itself is again a member of $\mathcal{G}_{H}^{0}$;
  \item [$(ii)$] If $(f+g)/2 \in \mathcal{G}$ for all $f$, $g \in \mathcal{G}$, then $\mathcal{G}_{H}^{0}$ is closed under convolution.
\end{itemize}
\end{theorem}

\begin{proof}
Let $f=h+\bar{g} \in \mathcal{G}_{H}^{0}$. To prove $(i)$, it suffices to show that $(h*h)+\epsilon (g*g) \in \mathcal{G}$ for each $|\epsilon|=1$. For $|\epsilon|=1$, note that
\[(h*h)+\epsilon (g*g)=(h+i \nu g)*(h-i \nu g)\]
where $\pm \nu$  are square roots of $\epsilon$. Since $\mathcal{G}$ is closed under convolution, it follows that $(h*h)+\epsilon (g*g) \in \mathcal{G}$ so that $f*f \in \mathcal{G}_{H}^{0}$. This proves $(i)$.

Regarding the proof of $(ii)$, let $f_i=h_i+\bar{g_i} \in \mathcal{G}_{H}^{0}$ ($i=1,2$). Considering the analytic functions
\begin{align*}
F_1&=(h_1-g_1)*(h_2-\epsilon g_2)\\
F_2&=(h_1+g_1)*(h_2+\epsilon g_2)
\end{align*}
for $|\epsilon|=1$, we see that
\[\frac{1}{2}(F_1+F_2)=(h_1*h_2)+\epsilon(g_1*g_2).\]
Since $F_1$, $F_2 \in \mathcal{G}$ and using the hypothesis, it is easy to deduce that $f_1*f_2 \in \mathcal{G}_{H}^{0}$.
\end{proof}

Theorem \ref{th2.9} immediately gives
\begin{corollary}\label{cor2.10}
Suppose that $\mathcal{G}\subset \mathcal{S}$ is a convex set and is closed under convolution. If $\mathcal{G}\triangleright \mathcal{G}_{H}^{0}$, then $\mathcal{G}_{H}^{0}$ is closed under convolution.
\end{corollary}

In \cite{goodloe}, Goodloe considered the Hadamard product $\tilde{*}$ of a harmonic function with an analytic function defined as follows:
\[f \tilde{*} \varphi=\varphi \tilde{*} f=h*\varphi+\overline{g*\varphi},\]
where $f=h+\bar{g}$ is harmonic and $\varphi$ is analytic in $\mathbb{D}$. The next theorem investigates the properties of the product $\tilde{*}$.

\begin{theorem}\label{th2.11}
Suppose that $\mathcal{G}\subset \mathcal{S}$ and $\mathcal{G}\triangleright \mathcal{G}_{H}^{0}$. Let $\mathcal{O}$ be a subfamily of $\mathcal{A}$ such that $\mathcal{G}$ is closed under convolution with members of $\mathcal{O}$. Then $\varphi \tilde{*} f \in \mathcal{G}_{H}^{0}$ for all $\varphi \in  \mathcal{O}$ and $f \in \mathcal{G}_{H}^{0}$.
\end{theorem}

\begin{proof}
Let $f=h+\bar{g} \in \mathcal{G}_{H}^{0}$ and $\varphi \in \mathcal{O}$. Then
\[\varphi\tilde{*}f=\varphi*h+\overline{\varphi*g}=H+\overline{G},\]
where $H=\varphi*h$ and $G=\varphi*g$ are analytic in $\mathbb{D}$. Setting $F=H+\epsilon G=\varphi*(h+\epsilon g)$ where $|\epsilon|=1$, we note that $F \in \mathcal{G}$ since $\mathcal{G}*\mathcal{O}\subset \mathcal{G}$. Thus $H+\overline{G} \in \mathcal{G}_{H}^{0}$ as desired.
\end{proof}

The next theorem indicates that the classes $\mathcal{G}$ and $\mathcal{G}_{H}^{0}$ have similar convex combination properties.
\begin{theorem}\label{th2.12}
Suppose that $\mathcal{G}\subset \mathcal{S}$ and $\mathcal{G}\triangleright \mathcal{G}_{H}^{0}$. Then $\mathcal{G}$ is closed under convex combinations if and only if $\mathcal{G}_{H}^{0}$ is closed under convex combinations.
\end{theorem}

\begin{proof}
Firstly we will prove the necessary part. For $n=1,2,\ldots$, suppose that $f_n \in \mathcal{G}_{H}^{0}$ where $f_n=h_n+\overline{g}_n$. For $\sum_{n=1}^{\infty}t_n=1$, $0\leq t_n \leq 1$, the convex combination of $f_n$'s may be written as
\[f(z)=\sum_{n=1}^{\infty}t_n f_n(z)=h(z)+\overline{g(z)},\]
where
\[h(z)=\sum_{n=1}^{\infty}t_n h_n(z)\quad \mbox{and} \quad g(z)=\sum_{n=1}^{\infty}t_n g_n(z).\]
are analytic in $\mathbb{D}$ with $h(0)=g(0)=h'(0)-1=g'(0)=0$. For $|\epsilon|=1$, we have
\[(h+\epsilon g)(z)=\sum_{n=1}^{\infty} t_n (h_n+\epsilon g_n )(z),\quad z \in \mathbb{D}.\]
Since the class $\mathcal{G}$ is closed under convex combination and $h_n+\epsilon g_n \in \mathcal{G}$ for $n=1,2,\ldots$, it follows that $h+\epsilon g \in \mathcal{G}$. Thus $f=h+\bar{g} \in \mathcal{G}_{H}^{0}$. This proves the necessary part.

The sufficient part follows by using the fact that $\mathcal{G} \subset \mathcal{G}_{H}^{0}$ and applying Theorem \ref{th2.2}$(ii)$.
\end{proof}

Theorem \ref{th2.12} immediately yields
\begin{corollary}\label{cor2.13}
Suppose that $\mathcal{G}\subset \mathcal{S}$ and $\mathcal{G}\triangleright \mathcal{G}_{H}^{0}$. Then $\mathcal{G}$ is a convex set if and only if $\mathcal{G}_{H}^{0}$ is a convex set.
\end{corollary}

\begin{remark}\label{rem2.14}
The harmonic Koebe function
\begin{equation}\label{eq2.2}
K(z)=H(z)+\overline{G(z)}, \quad H(z):=\frac{z-\frac{1}{2}z^2+\frac{1}{6}z^3}{(1-z)^3},\quad G(z):=\frac{\frac{1}{2}z^2+\frac{1}{6}z^3}{(1-z)^3}
\end{equation}
shows that the classes $\mathcal{S}_{H}^{0}$, $\mathcal{S}_{H}^{*0}$ and $\mathcal{C}_{H}^{0}$ are not harmonic analogues of any subfamily of $\mathcal{S}$ since
\[H(z)+G(z)=\frac{z+\frac{1}{3}z^3}{(1-z)^3},\quad z \in \mathbb{D},\]
and $(H+G)(i/\sqrt{3})=(H+G)(-i/\sqrt{3})$ which imply that $H+G$ is not univalent in $\mathbb{D}$. Similarly, $\mathcal{K}_{H}^{0}$ is not a harmonic analogue of any subfamily  $\mathcal{G} \subset \mathcal{S}$. For if $\mathcal{G} \triangleright  \mathcal{K}_{H}^{0}$ then $\mathcal{G}\subset \mathcal{K}$. The harmonic half-plane mapping
\begin{equation}\label{eq2.3}
L(z)=M(z)+\overline{N(z)},\quad  M(z):=\frac{z-\frac{1}{2}z^{2}}{(1-z)^2},\quad  N(z):=\frac{-\frac{1}{2}z^2}{(1-z)^2}
\end{equation}
belongs to $\mathcal{K}_{H}^{0}$ and $M(z)-N(z)=z/(1-z)^2 \not\in \mathcal{K}$.
\end{remark}

Remark \ref{rem2.14} suggests that given a subfamily $\mathcal{G}_{H}^{0} \subset \mathcal{S}_{H}^{0}$, it is possible that $\mathcal{G}_{H}^{0}$ is not a harmonic analogue of any subclass of $\mathcal{S}$. This motivates us to determine a necessary and sufficient condition for a subfamily $\mathcal{G}_{H}^{0}\subset \mathcal{S}_{H}^{0}$ to be a harmonic analogue of some family $\mathcal{G}\subset \mathcal{S}$. This is contained in the following theorem.

\begin{theorem}\label{th2.15}
A subfamily $\mathcal{G}_{H}^{0} \subset \mathcal{S}_{H}^{0}$ is a harmonic analogue of some family $\mathcal{G}\subset \mathcal{S}$ if and only if $\mathcal{G}_{H}^{0} \subset \mathcal{SS}_{H}^{0}$.
\end{theorem}

\begin{proof}
The necessary part follows by Theorem \ref{th2.2}$(i)$. For the sufficient part, suppose that $\mathcal{G}_{H}^{0} \subset \mathcal{SS}_{H}^{0}$. Consider the set
\[\mathcal{G}=\{h+ \epsilon g: h+\bar{g} \in \mathcal{G}_{H}^{0}\mbox{ and } |\epsilon|=1\}.\]
Using Lemma \ref{lem}, it is easily seen that $\mathcal{G} \subset \mathcal{S}$ and $\mathcal{G}\triangleright \mathcal{G}_{H}^{0}$.
\end{proof}

Keeping in mind that $\mathcal{S}\triangleright \mathcal{SS}_{H}^{0}$, $\mathcal{S}^{*} \triangleright \mathcal{SS}_{H}^{*0}$, $\mathcal{K} \triangleright\mathcal{SK}_{H}^{0}$ and $\mathcal{C}\triangleright \mathcal{SC}_{H}^{0}$, we determine the coefficient estimates, growth results, convolution properties and sharp bound for radius of starlikeness, convexity and close-to-convexity for the classes  $\mathcal{SS}_{H}^{0}$, $\mathcal{SS}_{H}^{*0}$, $\mathcal{SK}_{H}^{0}$ and $\mathcal{SC}_{H}^{0}$, using the results proved in this section. Note that parts $(i)$ and $(ii)$ of  the following theorem have been independently proved in \cite[Section 7]{stable}.

\begin{theorem}\label{th2.16}
Let $f=h+\bar{g} \in \mathcal{S}_{H}^{0}$ where $h$ and $g$ are given by \eqref{eq1.1}.
\begin{itemize}
  \item [$(i)$] (Coefficient estimates) If $f\in \mathcal{SS}_{H}^{0}, \mathcal{SS}_{H}^{*0}$ or $\mathcal{SC}_{H}^{0}$, then the sharp inequality $||a_n|-|b_n||\leq n$ holds for $n=2,3,\ldots$. Equality occurs for the analytic Keebe function $k(z)=z/(1-z)^2$. In case, $f \in \mathcal{SK}_{H}^{0}$ then $||a_n|-|b_n||\leq 1$ for $n=2,3,\ldots$, with the equality occurring for the analytic half-plane mapping $l(z)=z/(1-z)$.
  \item [$(ii)$] (Growth estimates and covering theorem) If $f\in \mathcal{SS}_{H}^{0}, \mathcal{SS}_{H}^{*0}$ or $\mathcal{SC}_{H}^{0}$, then we have
  \[\frac{|z|}{(1+|z|)^2} \leq |f(z)|\leq \frac{|z|}{(1-|z|)^2},\quad z \in \mathbb{D}.\]
  In particular, the range $f(\mathbb{D})$ contains the disk $|w|<1/4$. These results are sharp for the  analytic Koebe function $k$. If $f \in \mathcal{SK}_{H}^{0}$, then
  \[\frac{|z|}{1+|z|} \leq |f(z)|\leq \frac{|z|}{1-|z|},\quad z \in \mathbb{D},\]
  and therefore the range $f(\mathbb{D})$ contains the disk $|w|<1/2$. The analytic half plane mapping $l$ shows that these results are best possible.
  \item [$(iii)$] (Compactness) The classes $\mathcal{SS}_{H}^{0}$, $\mathcal{SS}_{H}^{*0}$, $\mathcal{SK}_{H}^{0}$ and $\mathcal{SC}_{H}^{0}$ are compact with respect to the topology of locally uniform convergence.
  \item [$(iv)$] (Radii of starlikeness, convexity and close-to-convexity) Let $r_{S}(\mathcal{G}_{H}^{0})$, $r_{C}(\mathcal{G}_{H}^{0})$ and $r_{CC}(\mathcal{G}_{H}^{0})$ denote the radius of starlikeness, convexity and close-to-convexity respectively of a subclass $\mathcal{G}_{H}^{0}\subset \mathcal{S}_{H}^{0}$. Then
      \[r_{S}(\mathcal{SS}_{H}^{*0})=r_{S}(\mathcal{SK}_{H}^{0})=r_{C}(\mathcal{SK}_{H}^{0})=
      r_{CC}(\mathcal{SS}_{H}^{*0})=r_{CC}(\mathcal{SK}_{H}^{0})=r_{CC}(\mathcal{SC}_{H}^{0})=1;\]
      \[r_{C}(\mathcal{SS}_{H}^{0})=r_{C}(\mathcal{SS}_{H}^{*0})=r_{C}(\mathcal{SC}_{H}^{0})=2-\sqrt{3};\]
      \[r_{S}(\mathcal{SS}_{H}^{0})=\tanh (\pi/4),\quad \mbox{and}\quad r_{S}(\mathcal{SC}_{H}^{0})=4\sqrt{2}-5.\]
      For $r_{CC}(\mathcal{SS}_{H}^{0})$, refer to \cite{krzyz}.
  \item [$(v)$] (Convolution properties)
  \begin{itemize}
    \item [$(a)$] If $f \in \mathcal{SK}_{H}^{0}$, then $f*f \in \mathcal{SK}_{H}^{0}$.
    \item [$(b)$] If $\varphi \in \mathcal{K}$ and $f \in \mathcal{SS}_{H}^{*0}$ (resp.  $f \in \mathcal{SK}_{H}^{0}$ and $f \in \mathcal{SC}_{H}^{0}$),  then $f \tilde{*} \varphi \in \mathcal{SS}_{H}^{*0}$ (resp.  $f \tilde{*} \varphi \in \mathcal{SK}_{H}^{0}$ and $f \tilde{*} \varphi \in \mathcal{SC}_{H}^{0}$).
  \end{itemize}
\end{itemize}
\end{theorem}

\begin{proof}
Making use of the well-known coefficient estimates and distortion theorems for functions in the class $\mathcal{S}$ (see \cite{goodman}), parts $(i)$ and $(ii)$ follow by applying Theorems \ref{th2.5} and \ref{th2.6} respectively. Theorem \ref{th2.7} gives $(iii)$, while $(iv)$  follows by using \cite[Chapter 13]{goodman} and Theorem \ref{th2.8}. Since $\mathcal{K}*\mathcal{S}^{*} \subset \mathcal{S}^{*}$,  $\mathcal{K}*\mathcal{K} \subset \mathcal{K}$ and $\mathcal{K}*\mathcal{C} \subset \mathcal{C}$, the convolution properties are easy to deduce from Theorems \ref{th2.9}$(i)$ and \ref{th2.11}.
\end{proof}

We close this section with the following remark.
\begin{remark}\label{rem2.17}
It is clear that the classes $\mathcal{SS}_{H}^{0}$, $\mathcal{SS}_{H}^{*0}$ and $\mathcal{SC}_{H}^{0}$ are not closed under convolution. However, since $\mathcal{K}*\mathcal{K} \subset \mathcal{K}$, $\mathcal{K} \triangleright\mathcal{SK}_{H}^{0}$ and $\mathcal{K}$ is non-convex, it is expected that $\mathcal{SK}_{H}^{0}$ is also not closed under convolution in view of Corollary \ref{cor2.10}. It will be an interesting open problem to determine whether $\mathcal{SK}_{H}^{0}$ is closed under convolution.
\end{remark}


\section{Harmonic analogues of subclasses of $\mathcal{S}$}\label{sec3}
In this section, we will determine the harmonic analogues of certain subclasses of $\mathcal{S}$. Apart from results of Section \ref{sec2}, we will make use of the following two lemmas which are the generalization of Theorems \ref{th2.9} and \ref{th2.11}. Their proof being similar are omitted.
\begin{lemma}\label{lem3.1}
Let $\mathcal{I}$ and $\mathcal{J}$ be subfamilies of $\mathcal{S}$ such that $\mathcal{I}*\mathcal{I}\subset \mathcal{J}$. If $\mathcal{I}_{H}^{0}$ and $\mathcal{J}_{H}^{0}$ denote the harmonic analogues of $I$ and $J$ respectively, then
\begin{itemize}
  \item [$(a)$] If $f \in \mathcal{I}_{H}^{0}$, then $f*f \in \mathcal{J}_{H}^{0}$;
  \item [$(b)$] If $(f+g)/2 \in \mathcal{J}$ for all $f$, $g \in \mathcal{J}$, then $\mathcal{I}_{H}^{0}*\mathcal{I}_{H}^{0} \subset \mathcal{J}_{H}^{0}$.
\end{itemize}
\end{lemma}

\begin{lemma}\label{lem3.2}
Suppose that $\mathcal{I}$ and $\mathcal{J}$ be subfamilies of $\mathcal{S}$. Let $\mathcal{O}\subset \mathcal{A}$ be such that $f*g \in \mathcal{J}$ for all $f \in \mathcal{I}$ and $g \in \mathcal{O}$. Then $\varphi \tilde{*} f \in \mathcal{J}_{H}^{0}$ for all $\varphi \in  \mathcal{O}$ and $f \in \mathcal{I}_{H}^{0}$, where $\mathcal{I}\triangleright\mathcal{I}_{H}^{0}$ and $\mathcal{J}\triangleright \mathcal{J}_{H}^{0}$.
\end{lemma}

Denote by $\mathcal{R}$ the class consisting of functions $f \in \mathcal{A}$ which satisfy $\RE f'(z)>0$ for $z \in \mathbb{D}$.  By well-known Noshiro-Warschawski Theorem (see \cite[Chapter 7, p. 88]{goodman}), $\mathcal{R}\subset \mathcal{S}$. In  \cite{macgregor1}, MacGregor investigated the properties of functions in the class $\mathcal{R}$. Also, it is easy to see that $\mathcal{R}$ is a compact family and is closed under convex combinations. However, the class $\mathcal{R}$ is not closed under convolutions. The analytic function
\begin{equation}\label{eq3.1}
f(z)=-z-2\log(1-z)=z+\sum_{n=2}^{\infty}\frac{2}{n}z^{n}
\end{equation}
belongs to $\mathcal{R}$ but $f*f \not\in \mathcal{R}$. The first theorem of this section determines the harmonic analogue of the class $\mathcal{R}$ and discusses its properties.

\begin{theorem}\label{th3.3}
The class $\mathcal{R}_{H}^{0}$ is the harmonic analogue of $\mathcal{R}$ where
\[\mathcal{R}_{H}^{0}=\{f=h+\bar{g}\in \mathcal{H}: \RE h'(z)>|g'(z)| \mbox{ for all } z \in \mathbb{D}\}.\]
In particular, $\mathcal{R}_{H}^{0} \subset \mathcal{SC}_{H}^{0}$. Moreover, we have
\begin{itemize}
  \item [$(i)$] If $f=h+\bar{g} \in \mathcal{R}_{H}^{0}$ where $h$ and $g$ are given by \eqref{eq1.1}, then $||a_n|-|b_n||\leq 2/n$ for $n=2,3,\ldots$. Equality holds for the function $f$ given by \eqref{eq3.1}.
  \item [$(ii)$] Every function $f \in \mathcal{R}_{H}^{0}$ satisfies
  \[-|z|+2\log(1+|z|)\leq |f(z)|\leq -|z|-2\log(1-|z|), \quad z \in \mathbb{D}\]
  and hence the range of each function $f \in \mathcal{R}_{H}^{0}$ contains the disk $|w|<2\log 2-1$. These results are sharp for the function $f$ given by \eqref{eq3.1}.
  \item [$(iii)$] The class $\mathcal{R}_{H}^{0}$ is compact with respect to the topology of locally uniform convergence.
  \item [$(iv)$] $r_{C}(\mathcal{R}_{H}^{0})=\sqrt{2}-1$ and $r_{CC}(\mathcal{R}_{H}^{0})=1$.
  \item [$(v)$] If $\varphi \in \mathcal{K}$ and $f \in \mathcal{R}_{H}^{0}$, then $f \tilde{*} \varphi \in \mathcal{R}_{H}^{0}$. Also, if $f \in \mathcal{A}$ with $\RE \varphi(z)/z>1/2$ for $z \in \mathbb{D}$ and $f \in \mathcal{R}_{H}^{0}$, then $f \tilde{*} \varphi \in \mathcal{R}_{H}^{0}$.
  \item [$(vi)$] The class $\mathcal{R}_{H}^{0}$ is closed under convex combinations of its members.
\end{itemize}
\end{theorem}

\begin{proof}
Suppose that $\mathcal{R} \triangleright \mathcal{G}_{H}^{0}$. If $f=h+\bar{g} \in \mathcal{G}_{H}^{0}$, then the inequality $\RE (h'(z)+\epsilon g'(z))>0$ holds for each $z \in \mathbb{D}$ and $|\epsilon|=1$. With appropriate choice of $\epsilon=\epsilon(z)$, it follows that
\[\RE h'(z)>|g'(z)|,\quad z\in \mathbb{D}\]
so that $f \in \mathcal{R}_{H}^{0}$. To prove the reverse inclusion, let $f=h+\bar{g} \in \mathcal{R}_{H}^{0}$. Then for $|\epsilon|=1$ we have
\[\RE (h'(z)+\epsilon g'(z))\geq \RE h'(z)-|g'(z)|>0,\quad z \in \mathbb{D}\]
which imply that $h+\epsilon g \in \mathcal{R}$ and hence $f \in \mathcal{G}_{H}^{0}$. This shows that $\mathcal{R}\triangleright \mathcal{R}_{H}^{0}$.

Since $\mathcal{R} \subset \mathcal{C}$, $\mathcal{R}_{H}^{0} \subset \mathcal{SC}_{H}^{0}$ by Theorem \ref{th2.2}$(iv)$. In view of \cite[Theorems 1 and 2, p. 533]{macgregor1}, the proof of parts $(i)$, $(ii)$ and $(iv)$ follow by applying Theorems \ref{th2.5}, \ref{th2.6} and \ref{th2.8} respectively. Theorems \ref{th2.7} and \ref{th2.12} verify the validity of $(iii)$ and $(vi)$ respectively. Since $\mathcal{K}* \mathcal{R} \subset \mathcal{R}$ (by \cite[Corollary 3.10]{naveen}), Theorem \ref{th2.11} shows that $f \tilde{*} \varphi \in \mathcal{R}_{H}^{0}$ if $\varphi \in \mathcal{K}$ and $f \in \mathcal{R}_{H}^{0}$. Regarding the proof of the other part of $(v)$, it suffices to show that if $\varphi \in \mathcal{A}$ with $\RE \varphi(z)/z>1/2$ and $f \in \mathcal{R}$, then $f * \varphi \in \mathcal{R}$. To see this, note that $(f*\varphi)'(z)=f'(z)*\varphi(z)/z$ for $z \in \mathbb{D}$. By \cite[Lemma 4, p.146]{singh2}, it follows that $\RE (f*\varphi)'>0$ so that $f * \varphi \in \mathcal{R}$. This concludes the theorem.
\end{proof}

Note that Mocanu \cite{mocanu} independently proved that if $f$ is a harmonic mapping in a convex domain $\Omega$ such that $\RE f_{z}(z)>|f_{\bar{z}}(z)|$ for $z \in \Omega$, then $f$ is univalent and sense-preserving in $\Omega$ while Ponnusamy \emph{et al.} \cite{ponnusamy} showed that members of $\mathcal{R}_{H}^{0}$ are close-to-convex in $\mathbb{D}$.

In \cite{chichra}, Chichra introduced the class $\mathcal{W}$ of analytic functions $f \in \mathcal{A}$ which satisfy $\RE(f'(z)+zf''(z))>0$ for $z \in \mathbb{D}$. He proved that the members of $\mathcal{W}$ are univalent in $\mathbb{D}$ by showing that $\mathcal{W}\subset \mathcal{R}$. Later R. Singh and S. Singh \cite{singh1} proved that $\mathcal{W} \subset \mathcal{S}^{*}$. The class $\mathcal{W}$ is compact and is closed under convex combination of its members. Similar to the proof of Theorem \ref{th3.3}, it can be shown that the set
\[\mathcal{W}_{H}^{0}=\{f=h+\bar{g}\in \mathcal{H}: \RE (h'(z)+zh''(z))>|g'(z)+zg''(z)| \mbox{ for all } z \in \mathbb{D}\}.\]
is the harmonic analogue of $\mathcal{W}$. By Theorem \ref{th2.2}$(iv)$, $\mathcal{W}_{H}^{0} \subset \mathcal{R}_{H}^{0}\cap \mathcal{SS}_{H}^{*0}$. In particular, the members of $\mathcal{W}_{H}^{0}$ are fully starlike in $\mathbb{D}$ by Corollary \ref{cor2.4}. To determine the coefficient and growth estimates for functions in the class $\mathcal{W}_{H}^{0}$, we need to prove the following lemma.

\begin{lemma}\label{lem3.4}
If $f \in \mathcal{W}$ is given by \eqref{eq2.1}, then $|a_n|\leq 2/n^2$ for $n=2,3,\ldots$ and
\[-1+\frac{2}{|z|}\log(1+|z|)\leq|f'(z)| \leq -1-\frac{2}{|z|}\log(1-|z|),\quad z \in \mathbb{D}.\]
The function
\begin{equation}\label{eq3.2}
f(z)=-z-2\int_{0}^{|z|}\frac{1}{t}\log (1-t)\,dt=z+\sum_{n=2}^{\infty}\frac{2}{n^2}z^n
\end{equation}
shows that all these results are sharp.
\end{lemma}

\begin{proof}
Observe that $f \in \mathcal{W}$ if and only if $z f' \in \mathcal{R}$. The proof now follows by applying \cite[Theorem 1, p. 533]{macgregor1}.
\end{proof}

\begin{theorem}\label{th3.5}
Let $f=h+\bar{g} \in \mathcal{W}_{H}^{0}$ where $h$ and $g$ are given by \eqref{eq1.1}. Then $||a_n|-|b_n||\leq 2/n^2$ for $n=2,3,\ldots$ and
\[-|z|+2\int_{0}^{|z|}\frac{1}{t}\log(1+t)\,dt\leq|f(z)| \leq -|z|-2\int_{0}^{|z|}\frac{1}{t}\log(1-t)\,dt,\quad z \in \mathbb{D}.\]
In particular, the range $f(\mathbb{D})$ contains the disk $|w|<\pi^2/6-1$. All these results are sharp for the function $f$ given by \eqref{eq3.2}. Moreover, the following statements regarding the class $\mathcal{W}_{H}^{0}$ hold.
\begin{itemize}
  \item [$(i)$] The class $\mathcal{W}_{H}^{0}$ is compact with respect to the topology of locally uniform convergence.
  \item [$(ii)$] $r_{S}(\mathcal{W}_{H}^{0})=1=r_{CC}(\mathcal{W}_{H}^{0})$.
  \item [$(iii)$] The class $\mathcal{W}_{H}^{0}$ is closed under convolutions.
  \item [$(iv)$] \begin{itemize}
                   \item [$(a)$] If $\varphi \in \mathcal{K}$ and $f \in \mathcal{W}_{H}^{0}$, then $\varphi\tilde{*}f  \in \mathcal{W}_{H}^{0}$;
                   \item [$(b)$] If $\varphi \in \mathcal{A}$ with $\RE \varphi(z)/z>1/2$ and $f \in \mathcal{W}_{H}^{0}$, then $\varphi\tilde{*}f  \in \mathcal{W}_{H}^{0}$;
                   \item [$(c)$] If $\varphi \in \mathcal{W}$ and $f \in \mathcal{W}_{H}^{0}$, then $\varphi\tilde{*}f  \in \mathcal{W}_{H}^{0}\cap \mathcal{SK}_{H}^{0}$.
                 \end{itemize}
  \item [$(v)$] The class $\mathcal{W}_{H}^{0}$ is closed under convex combinations.
\end{itemize}
\end{theorem}

\begin{proof}
The growth and coefficient estimates for the class $\mathcal{W}_{H}^{0}$ follow by Lemma \ref{lem3.4}. Since $\mathcal{W}$ is convex and closed under convolutions (see \cite{singh2}), the class $\mathcal{W}_{H}^{0}$ is closed under convolutions by Corollary \ref{cor2.10}. This proves $(iii)$. Since $\mathcal{W}_{H}^{0} \subset \mathcal{SS}_{H}^{*0}$, $(ii)$ is obviously true. To prove $(iv)$, note that $\mathcal{K}*\mathcal{W}\subset \mathcal{W}$ (by \cite[Corollary 3.10]{naveen}) and if $\varphi \in \mathcal{A}$ with $\RE \varphi(z)/z>1/2$ and $f \in \mathcal{W}$, then $f * \varphi \in \mathcal{W}$ (by \cite[Theorem 3', p. 150]{singh2}). These observations lead to $(a)$ and $(b)$ by applying Theorem \ref{th2.11}. Since $\mathcal{W}*\mathcal{W}\subset \mathcal{W}\cap \mathcal{K}$ (by \cite{singh2}), part $(c)$ follows by Lemma \ref{lem3.2}. Theorems \ref{th2.7} and \ref{th2.12} verify the validity of the parts $(i)$ and $(v)$ respectively.
\end{proof}

\begin{remark}\label{rem3.6}
Since $\mathcal{W}*\mathcal{W}\subset \mathcal{K}$, the convolution of each member of $\mathcal{W}_{H}^{0}$ with itself is convex in $\mathbb{D}$ by Lemma \ref{lem3.1}$(a)$. However, since $\mathcal{K}$ is non-convex, it is not known whether $\mathcal{W}_{H}^{0}*\mathcal{W}_{H}^{0} \subset \mathcal{SK}_{H}^{0}$ in view of Lemma \ref{lem3.1}$(b)$.
\end{remark}

Let $\mathcal{U}$ and $\mathcal{V}$ be subclasses of $\mathcal{A}$ consisting of functions $f$ of the form \eqref{eq2.1} that satisfy
\[\sum_{n=2}^{\infty}n|a_n|\leq 1\quad \mbox{and}\quad \sum_{n=2}^{\infty}n^2 |a_n|\leq 1\]
respectively. Clearly $\mathcal{V}\subset \mathcal{U}$. In \cite{good}, Goodman proved that $\mathcal{U}\subset \mathcal{S}^{*}$ and $\mathcal{V}\subset \mathcal{K}$. It is easy to see that $\mathcal{U}\subset \mathcal{R}$ and $\mathcal{V}\subset \mathcal{W}$. In fact, if $f \in \mathcal{U}$ is given by \eqref{eq2.1}, then
\[\RE f'(z)=1+\RE\sum_{n=2}^{\infty}n a_{n}z^{n-1}>1-\sum_{n=2}^{\infty}n |a_n|>0.\]
Similarly, if $f \in \mathcal{V}$ is given by \eqref{eq2.1}, then
\[\RE (f'(z)+z f''(z))=1+\RE\sum_{n=2}^{\infty}n^2 a_{n}z^{n-1}>1-\sum_{n=2}^{\infty}n^2 |a_n|>0.\]
The next theorem determines the harmonic analogue of the classes $\mathcal{U}$ and $\mathcal{V}$.

\begin{theorem}\label{th3.7}
The harmonic analogues of the classes $\mathcal{U}$ and $\mathcal{V}$ are given by
\[\mathcal{U}_{H}^{0}=\left\{f(z)=z+\sum_{n=2}^{\infty} a_n z^n+\overline{\sum_{n=2}^{\infty}b_n z^n}\in \mathcal{H}:\sum_{n=2}^{\infty} n(|a_n|+|b_n|)\leq 1\right\}\]
and
\[\mathcal{V}_{H}^{0}=\left\{f(z)=z+\sum_{n=2}^{\infty} a_n z^n+\overline{\sum_{n=2}^{\infty}b_n z^n}\in \mathcal{H}:\sum_{n=2}^{\infty} n^2 (|a_n|+|b_n|)\leq 1\right\}\]
respectively.
\end{theorem}

\begin{proof}
Suppose that $\mathcal{U}\triangleright \mathcal{G}_{H}^{0}$. If $f=h+\bar{g} \in \mathcal{G}_{H}^{0}$ where $h$ and $g$ are given by \eqref{eq1.1}, then $h+\epsilon g \in \mathcal{U}$ for each $|\epsilon|=1$ so that
\[\sum_{n=2}^{\infty} n|a_n +\epsilon b_n|\leq 1.\]
On choosing $\epsilon=\epsilon (n)$ wisely we deduce that $f \in \mathcal{U}_{H}^{0}$. Conversely if $f=h+\bar{g} \in \mathcal{U}_{H}^{0}$ where $h$ and $g$ are given by \eqref{eq1.1}, then for $|\epsilon|=1$ we have
\[\sum_{n=2}^{\infty} n|a_n +\epsilon b_n|\leq \sum_{n=2}^{\infty}n(|a_n|+|b_n|)\leq 1\]
so that $h+\epsilon g \in \mathcal{U}$ and hence $f\in \mathcal{G}_{H}^{0}$. Thus $\mathcal{U}\triangleright \mathcal{U}_{H}^{0}$. Similarly it can be shown that $\mathcal{V}\triangleright \mathcal{V}_{H}^{0}$.
\end{proof}

In view of Theorem \ref{th2.2}$(iv)$, $\mathcal{U}_{H}^{0} \subset \mathcal{SS}_{H}^{*0}$ and $\mathcal{V}_{H}^{0} \subset \mathcal{SK}_{H}^{0}$. In particular, the members of $\mathcal{U}_{H}^{0}$ (resp. $\mathcal{V}_{H}^{0}$) are fully starlike (resp. fully convex) in $\mathbb{D}$ by Corollary \ref{cor2.4}. Also since $\mathcal{U}\subset \mathcal{R}$ and $\mathcal{V}\subset \mathcal{W}$, therefore $\mathcal{U}_{H}^{0}\subset \mathcal{R}_{H}^{0}$ (see also \cite{ponnusamy}) and $\mathcal{V}_{H}^{0}\subset \mathcal{W}_{H}^{0}$. Using the results of \cite{silverman1} and applying the theorems of Section \ref{sec2}, we have

\begin{corollary}\label{cor3.8}
Let $f=h+\bar{g} \in \mathcal{S}_{H}^{0}$ where $h$ and $g$ are given by \eqref{eq1.1}.
\begin{itemize}
  \item [$(a)$] If $f \in \mathcal{U}_{H}^{0}$, then
  \[|a_n|\leq 1/n,\quad|b_n|\leq 1/n\quad \mbox{and}\quad ||a_n|-|b_n||\leq 1/n \quad \mbox{for} \mbox{ n}=2,3,\ldots.\]
  Equality occurs for the functions $z+z^2/2$ and $z+\bar{z}^{2}/2$. If $f \in \mathcal{V}_{H}^{0}$, then the sharp inequalities
  \[|a_n|\leq 1/n^2,\quad|b_n|\leq 1/n^2\quad \mbox{and}\quad ||a_n|-|b_n||\leq 1/n^2 \quad \mbox{for} \mbox{ n}=2,3,\ldots\]
  hold with the equality occurring for the functions $z+z^2/4$ and $z+\bar{z}^{2}/4$.
  \item [$(b)$] If $f \in \mathcal{U}_{H}^{0}$, then
  \[|z|-\frac{1}{2}|z|^2 \leq |f(z)|\leq |z|+\frac{1}{2}|z|^2,\quad z \in \mathbb{D}.\]
  In particular, the range $f(\mathbb{D})$ contains the disc $|w|<1/2$. If $f \in \mathcal{V}_{H}^{0}$, then
  \[|z|-\frac{1}{4}|z|^2 \leq |f(z)|\leq |z|+\frac{1}{4}|z|^2,\quad z \in \mathbb{D},\]
  and therefore $f(\mathbb{D})$ contains the disk $|w|<3/4$.
  \item [$(c)$] The classes $\mathcal{U}_{H}^{0}$ and $\mathcal{V}_{H}^{0}$ are compact with respect to the topology of locally uniform convergence.
  \item [$(d)$] $r_{S}(\mathcal{U}_{H}^{0})=r_{CC}(\mathcal{U}_{H}^{0})=r_{S}(\mathcal{V}_{H}^{0})=r_{C}(\mathcal{V}_{H}^{0})=r_{CC}(\mathcal{V}_{H}^{0})=1$ and $r_{C}(\mathcal{U}_{H}^{0})=1/2$.
  \item [$(e)$] The classes $\mathcal{U}_{H}^{0}$ and $\mathcal{V}_{H}^{0}$ are closed under convex combinations.
\end{itemize}
\end{corollary}

Avci and Zlotkiewicz \cite{avci} independently investigated certain properties of the classes $\mathcal{U}_{H}^{0}$ and $\mathcal{V}_{H}^{0}$ (see also \cite{silverman2}). The next theorem investigates the convolution properties of the classes $\mathcal{U}_{H}^{0}$ and $\mathcal{V}_{H}^{0}$.

\begin{theorem}\label{th3.9}
The classes $\mathcal{U}_{H}^{0}$ and $\mathcal{V}_{H}^{0}$ are closed under convolutions. Moreover, we have
\begin{itemize}
  \item [$(i)$] $\mathcal{U}_{H}^{0}*\mathcal{U}_{H}^{0}\subset \mathcal{SK}_{H}^{0}$;
  \item [$(ii)$] If $\varphi \in \mathcal{K}$ and $f \in \mathcal{U}_{H}^{0}$, then $\varphi\tilde{*}f  \in \mathcal{U}_{H}^{0}$;
  \item [$(iii)$] If $\varphi \in \mathcal{K}$ and $f \in \mathcal{V}_{H}^{0}$, then $\varphi\tilde{*}f  \in \mathcal{V}_{H}^{0}$.
\end{itemize}
\end{theorem}

\begin{proof}
The main crux of the proof relies on the observation that if $f \in \mathcal{V}$ is given by \eqref{eq2.1}, then $\sum_{n=2}^{\infty} n^2 |a_n|^2\leq 1$. Since $\mathcal{U}$ and $\mathcal{V}$ are convex sets therefore it suffices to show that the classes $\mathcal{U}$ and $\mathcal{V}$ are closed under convolution in view of Corollary \ref{cor2.10}. Let $f$, $F \in \mathcal{V}$ be given by \eqref{eq2.1}. Then
\[\sum_{n=2}^{\infty}n^2 |a_n A_n|\leq \frac{1}{2}\sum_{n=2}^{\infty} n^2 |a_n|^2+\frac{1}{2}\sum_{n=2}^{\infty} n^2 |A_n|^2\leq 1\]
using the fact that the geometric mean is less than or equal to the arithmetic mean. This shows that $f *F \in \mathcal{V}$. The same calculation shows that if $f$, $F \in \mathcal{U}$, then $f *F \in \mathcal{V} \subset \mathcal{U}$.

The proof of part $(i)$ follows by Lemma \ref{lem3.1}$(b)$ since $\mathcal{U}*\mathcal{U}\subset \mathcal{V}$, $\mathcal{V}$ is a convex set and $\mathcal{V}\subset \mathcal{K}$. Since the classes $\mathcal{U}$ and $\mathcal{V}$ are closed under convolution with convex functions, $(ii)$ and $(iii)$ follows immediately from Theorem \ref{th2.11}.
\end{proof}

Let $\mathcal{S}_{\mathbb{R}}$ be the subclass of $\mathcal{S}$ consisting of functions $f$ of the form \eqref{eq2.1} whose coefficients $a_n$ are all real. We close this section by determining its harmonic analogue.

\begin{theorem}\label{th3.10}
$\mathcal{S}_{\mathbb{R}}$ is the harmonic analogue of $\mathcal{S}_{\mathbb{R}}$ itself.
\end{theorem}

\begin{proof}
Suppose that $\mathcal{S}_{\mathbb{R}} \triangleright \mathcal{G}_{H}^{0}$. Then $\mathcal{S}_{\mathbb{R}} \subset \mathcal{G}_{H}^{0}$. To prove the reverse inclusion, let $f=h+\bar{g}\in \mathcal{G}_{H}^{0}$ where $h$ and $g$ are given by \eqref{eq1.1}. Then $h+\epsilon g \in \mathcal{S}_{\mathbb{R}}$ for each $|\epsilon|=1$ which imply that all the coefficients $a_n+\epsilon b_n$ are real for each $|\epsilon|=1$. But this is possible only if $a_n$ are real and $b_n=0$ for $n=2,3,\ldots$. Thus $g \equiv 0$ and $f \in \mathcal{S}_{\mathbb{R}}$. Hence $\mathcal{S}_{\mathbb{R}}\triangleright \mathcal{S}_{\mathbb{R}}$.
\end{proof}


\section{Harmonic Alexander operator and Radius problems}
In this section, we will introduce the notion of harmonic Alexander operator and discuss some of its properties. We will also determine the radius of convexity for certain families of harmonic functions.

\begin{definition}\label{def4.1}
Define an integral operator $\Lambda_{H}^{+}:\mathcal{H}\rightarrow\mathcal{H}$ by
\[\Lambda_{H}^{+}[f]=\Lambda[h]+\overline{\Lambda[g]},\quad f=h+\bar{g}\in \mathcal{H},\]
where $\Lambda$ is the Alexander operator defined by \eqref{eq1.2}. We call $\Lambda_{H}^{+}$ the \textbf{positive harmonic Alexander operator}.
\end{definition}

\begin{figure}[here]
\centering
\includegraphics[width=3in]{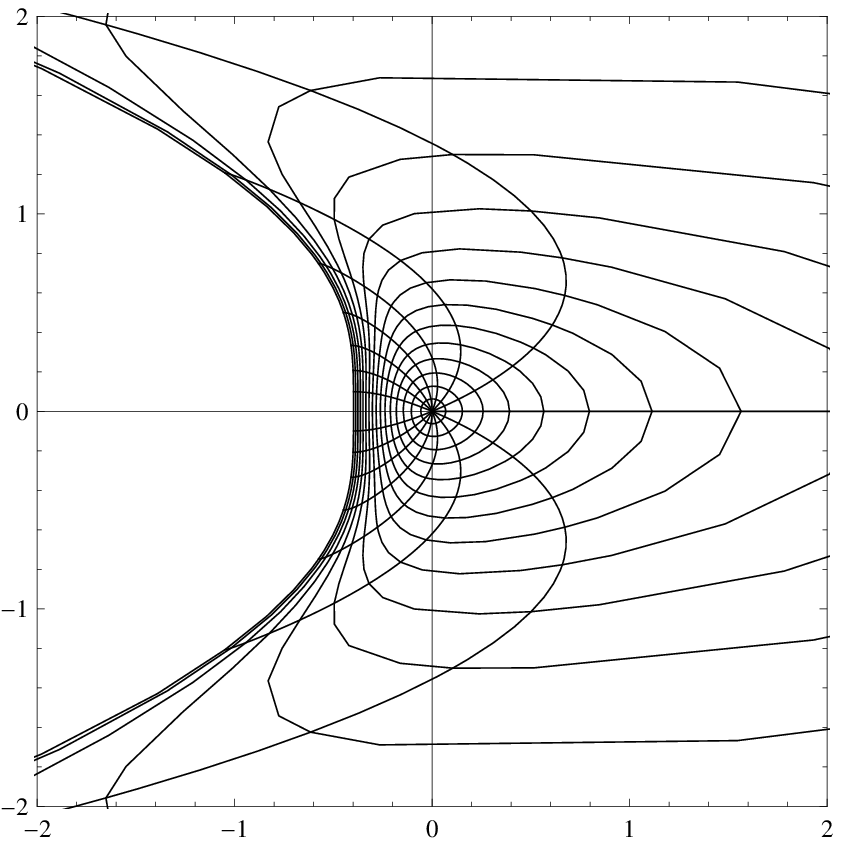}
\caption{Graph of the function $\Lambda_{H}^{+}[K]$.}\label{fig1}
\end{figure}

Since $\Lambda$ is linear therefore so is the operator $\Lambda_{H}^{+}$, that is,  $\Lambda_{H}^{+}[f_1+f_2]=\Lambda_{H}^{+}[f_1]+\Lambda_{H}^{+}[f_2]$ for all $f_1$, $f_2 \in \mathcal{H}$. The first theorem shows that if a subfamily $\mathcal{G} \subset \mathcal{S}$ is preserved under $\Lambda$, then its harmonic analogue $\mathcal{G}_{H}^{0}$ is preserved under $\Lambda_{H}^{+}$.

\begin{theorem}\label{th4.2}
Let $\mathcal{I}$ and $\mathcal{J}$ be subfamilies of $\mathcal{S}$ such that $\Lambda[\mathcal{I}]\subset \mathcal{J}$. Then $\Lambda_{H}^{+}[\mathcal{I}_{H}^{0}]\subset \mathcal{J}_{H}^{0}$ where $\mathcal{I}\triangleright\mathcal{I}_{H}^{0}$ and $\mathcal{J}\triangleright \mathcal{J}_{H}^{0}$.
\end{theorem}

\begin{proof}
Let  $f=h+\bar{g}\in \mathcal{I}_{H}^{0}$. Since $\Lambda_{H}^{+}[f]=\Lambda[h]+\overline{\Lambda[g]}$ and $\mathcal{J}\triangleright \mathcal{J}_{H}^{0}$, it suffices to show that $\Lambda[h]+\epsilon \Lambda[g] \in \mathcal{J}$ for each $|\epsilon|=1$. But $\Lambda[h]+\epsilon \Lambda[g]=\Lambda[h+\epsilon g]\in\Lambda[\mathcal{I}] \subset \mathcal{J}$ since $\mathcal{I}\triangleright\mathcal{I}_{H}^{0}$.
\end{proof}

Note that $\mathcal{R}_{H}^{0}\not\subset \mathcal{S}_{H}^{*0}$ and $\mathcal{U}_{H}^{0}\not\subset \mathcal{K}_{H}^{0}$. Since $\Lambda[\mathcal{R}]\subset \mathcal{W}\subset \mathcal{S}^{*}$ and $\Lambda[\mathcal{U}]\subset \mathcal{V}\subset \mathcal{K}$ Theorem \ref{th4.2} gives the following two corollaries.
\begin{corollary}\label{cor4.3}
$\Lambda_{H}^{+}[\mathcal{R}_{H}^{0}]\subset \mathcal{SS}_{H}^{*0}$ and $\Lambda_{H}^{+}[\mathcal{U}_{H}^{0}]\subset \mathcal{SK}_{H}^{0}$
\end{corollary}

\begin{corollary}\label{cor4.4}
The classes $\mathcal{R}_{H}^{0}$, $\mathcal{W}_{H}^{0}$, $\mathcal{U}_{H}^{0}$ and $\mathcal{V}_{H}^{0}$ are preserved under $\Lambda_{H}^{+}$.
\end{corollary}

The Alexander operator $\Lambda$ provides a one-to-one correspondence between the classes $\mathcal{S}^{*}$ and $\mathcal{K}$: $f \in \mathcal{S}^{*}$ if and only if $\Lambda[f] \in \mathcal{K}$. A similar result holds for the positive harmonic Alexander operator which provides a one-to-one correspondence between the classes $\mathcal{SS}_{H}^{*0}$ and $\mathcal{SK}_{H}^{0}$.
\begin{corollary}\label{cor4.5}
$f \in \mathcal{SS}_{H}^{*0}$ if and only if $\Lambda_{H}^{+}[f] \in \mathcal{SK}_{H}^{0}$.
\end{corollary}

However, the inclusion $\Lambda_{H}^{+}[\mathcal{S}_{H}^{*0}]\subset \mathcal{K}_{H}^{0}$ is not valid in general. To see this, note that the harmonic Koebe function $K$ given by \eqref{eq2.2} belongs to $\mathcal{S}_{H}^{*0}$ and
\[\Lambda_{H}^{+}[K](z)=\frac{1}{6}\left[\frac{z(5-3z)}{(1-z)^2}-\log (1-z)\right]+\overline{\frac{1}{6}\left[\frac{z(3z-1)}{(1-z)^2}-\log (1-z)\right]}.\]
The graph of the function $\Lambda_{H}^{+}[K]$ (see Figure \ref{fig1}) shows that the image domain is not even starlike. In particular, $\Lambda_{H}^{+}[\mathcal{S}_{H}^{*0}]\not\subset \mathcal{S}_{H}^{*0}$. Similarly, it can be shown that $\Lambda_{H}^{+}[\mathcal{K}_{H}^{0}]\not\subset \mathcal{K}_{H}^{0}$ by considering the harmonic half-plane mapping $L$ given by \eqref{eq2.3}. Note that
\begin{align*}
\Lambda_{H}^{+}[L](z)&=\frac{1}{2}\left[-\log(1-z)+\frac{z}{1-z}\right]+\overline{\frac{1}{2}\left[-\log(1-z)-\frac{z}{1-z}\right]}\\
                     &=-\log |1-z|+i \IM \left(\frac{z}{1-z}\right).
\end{align*}
Clearly Figure \ref{fig2} depicts that the image domain $\Lambda_{H}^{+}[L](\mathbb{D})$ is not convex. The failure of the implication $\Lambda_{H}^{+}[\mathcal{S}_{H}^{*0}]\subset \mathcal{K}_{H}^{0}$ motivates to introduce the following definition.

\begin{figure}[here]
\centering
\includegraphics[width=3in]{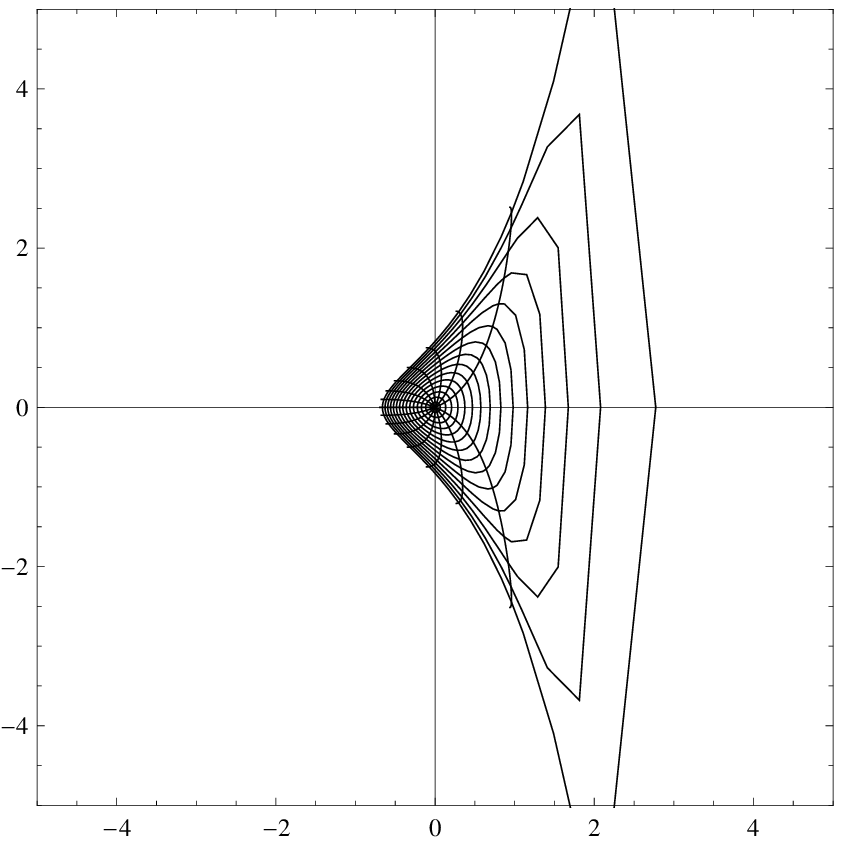}
\caption{Graph of the function $\Lambda_{H}^{+}[L]$.}\label{fig2}
\end{figure}

\begin{definition}\label{def4.6}
Define another integral operator $\Lambda_{H}^{-}:\mathcal{H}\rightarrow\mathcal{H}$ by
\[\Lambda_{H}^{-}[f]=\Lambda[h]-\overline{\Lambda[g]},\quad f=h+\bar{g}\in \mathcal{H},\]
where $\Lambda$ is given by \eqref{eq1.2}. We call $\Lambda_{H}^{-}$ the \textbf{negative harmonic Alexander operator}.
\end{definition}

By \cite[Lemma, p. 108]{duren} it follows that $\Lambda_{H}^{-}[\mathcal{S}_{H}^{*0}]\subset \mathcal{K}_{H}^{0}$. In particular, the classes $\mathcal{S}_{H}^{*0}$ and $\mathcal{K}_{H}^{0}$ are preserved under the operator $\Lambda_{H}^{-}$. It's worth to remark that Theorems \ref{th4.2} and its corollaries continue to hold for the negative harmonic Alexander operator $\Lambda_{H}^{-}$.

For $G \in \mathcal{A}$, consider the family
\[\mathcal{R}_{H}^{0}(G)=\left\{f=h+\bar{g}\in \mathcal{H}:\RE\frac{h'(z)}{G'(z)}>\left|\frac{g'(z)}{G'(z)}\right|\mbox{ for all } z\in \mathbb{D}\right\}.\]
If $G(z)=z$, then $\mathcal{R}_{H}^{0}(G)$ reduces to $\mathcal{R}_{H}^{0}$. In \cite{kaliraj}, it has been proved that if $G \in \mathcal{K}$, then $\mathcal{R}_{H}^{0}(G) \subset \mathcal{SC}_{H}^{0}$ (see also \cite{mocanu,mocanu1}). The next theorem determines the radius of convexity of the class $\mathcal{R}_{H}^{0}(G)$ for specific choices of the function $G$.

\begin{theorem}\label{th4.7}
Let $r_{C}$ denotes the radius of convexity of the class $\mathcal{R}_{H}^{0}(G)$ for $G \in \mathcal{A}$.
\begin{itemize}
  \item [$(i)$] If $G \in \mathcal{S}$, then $r_{C}=3-2\sqrt{2}$;
  \item [$(ii)$] If $G \in \mathcal{S}^{*}$, then $r_{C}=3-2\sqrt{2}$;
  \item [$(iii)$] If $G \in \mathcal{K}$, then $r_{C}=2-\sqrt{3}$;
  \item [$(iv)$] If $G \in \mathcal{R}$, then $r_{C}=\sqrt{5}-2$;
  \item [$(v)$] If $G \in \mathcal{A}$ with $\RE G'(z)>1/2$, then $r_{C}=3-2\sqrt{2}$.
\end{itemize}
Moreover, all these results are sharp.
\end{theorem}

\begin{proof}
Let $f=h+\bar{g}\in \mathcal{R}_{H}^{0}(G)$. Setting $F_{\epsilon}=h+\epsilon g$ for $|\epsilon|=1$ note that
\[\RE \frac{F_{\epsilon}'(z)}{G'(z)}=\RE\left(\frac{h'(z)}{G'(z)}+\epsilon\frac{g'(z)}{G'(z)}\right)\geq \RE \frac{h'(z)}{G'(z)}-\left|\frac{g'(z)}{G'(z)}\right|>0, \quad z \in \mathbb{D}.\]
If $G \in \mathcal{S}$, then $F_{\epsilon}$ is convex in $|z|<3-2\sqrt{2}$ by \cite[Theorem 1, p. 32]{ratti1} for each $|\epsilon|=1$. By Lemma \ref{lem}, $f$ is convex in $|z|<3-2\sqrt{2}$. This proves $(i)$. The proof of the other parts is similar.
\end{proof}

For $G \in \mathcal{A}$, let $\mathcal{F}_{H}^{0}(G)$ be the subclass of $\mathcal{R}_{H}^{0}(G)$ defined by the set
\[\mathcal{F}_{H}^{0}(G)=\left\{f=h+\bar{g}\in \mathcal{H}:\left|\frac{h'(z)}{G'(z)}-1\right|<1-\left|\frac{g'(z)}{G'(z)}\right|\mbox{ for all } z\in \mathbb{D}\right\}.\]
If $G(z)=z$, then $\mathcal{F}_{H}^{0}(G)$ reduces to $\mathcal{F}_{H}^{0}$. Also, $\mathcal{F}_{H}^{0}(G)\subset \mathcal{SC}_{H}^{0}$ if $G \in \mathcal{K}$. The next theorem is similar to that of Theorem \ref{th4.7}.

\begin{theorem}\label{th4.8}
Suppose that $r_{C}$ denotes the radius of convexity of the class $\mathcal{F}_{H}^{0}(G)$ for $G \in \mathcal{A}$.
\begin{itemize}
  \item [$(a)$] If $G \in \mathcal{S}$, then $r_{C}=1/5$;
  \item [$(b)$] If $G \in \mathcal{S}^{*}$, then $r_{C}=1/5$;
  \item [$(c)$] If $G \in \mathcal{K}$, then $r_{C}=1/3$;
  \item [$(d)$] If $G \in \mathcal{R}$, then $r_{C}=(\sqrt{17}-3)/4$;
  \item [$(e)$] If $G \in \mathcal{A}$ with $\RE G'(z)>1/2$, then $r_{C}$ is the smallest positive root of the equation $r^4+2r^3+13r^2+4r-4=0$.
\end{itemize}
Moreover, all these results are sharp.
\end{theorem}

\begin{proof}
If $f=h+\bar{g}\in \mathcal{F}_{H}^{0}(G)$, then it is easy to see that the analytic functions $F_{\epsilon}=h+\epsilon g$ ($|\epsilon|=1$) satisfy
\[\left| \frac{F_{\epsilon}'(z)}{G'(z)}-1\right|<1, \quad z \in \mathbb{D}.\]
The proof of the theorem now follows by using the results of \cite{ratti2} and applying Lemma \ref{lem}.
\end{proof}

\end{document}